\documentclass[10pt]{article}

\usepackage{amssymb,amsthm,amsmath}
\usepackage{enumerate}
\usepackage{graphicx,color}
\usepackage[hidelinks]{hyperref}

\newcommand{\dd}{\mathrm{d}}
\newcommand{\E}{\mathbb{E}}
\newcommand{\1}{\textbf{1}}
\newcommand{\R}{\mathbb{R}}

\newcommand{\p}[1]{\mathbb{P}\left( #1 \right)}
\newcommand{\scal}[2]{\left\langle #1, #2 \right\rangle}
\DeclareMathOperator{\conv}{conv}

\newcommand{\red}{}

\usepackage[paper=a4paper, left=1.5in, right=1.5in, top=1.5in, bottom=1.5in]{geometry}
\newtheorem{theorem}{Theorem}
\newtheorem{lemma}[theorem]{Lemma}

\theoremstyle{remark}
\newtheorem{remark}[theorem]{Remark}

\theoremstyle{definition}

\title{\vspace{-3em}A note on volume thresholds for random polytopes}
\author{
Debsoumya Chakraborti\thanks{Carnegie Mellon University; Pittsburgh, PA 15213, USA. Email: dchakrab@andrew.cmu.edu.} \ \ Tomasz Tkocz\thanks{Carnegie Mellon University; Pittsburgh, PA 15213, USA. Email: ttkocz@math.cmu.edu. Research supported in part by the Collaboration Grants from the Simons Foundation and NSF grant DMS-1955175.}
\ \
Beatrice-Helen Vritsiou\thanks{University of Alberta; Edmonton, AB T6G 2G1, Canada. Email:  vritsiou@ualberta.ca.}
}
\date{1st November 2020}

\begin{document}

\maketitle

\begin{abstract}
We study the expected volume of random polytopes generated by taking the convex hull of independent identically distributed points from a given distribution. We show that, for log-concave distributions supported on convex bodies, we need at least exponentially many (in dimension) samples for the expected volume to be significant, and that super-exponentially many samples suffice for {\red $\kappa$-}concave measures when their parameter of concavity {\red $\kappa$} is positive.
\end{abstract}

\bigskip

\begin{footnotesize}

\end{footnotesize}

\bigskip

\section{Introduction}

Let $X_1, X_2,\ldots$ be independent identically distributed (i.i.d.) random vectors uniform on a set $K$ in $\R^n$. Let
\begin{equation}\label{eq:def-KN}
K_N = \conv\{X_1,\ldots,X_N\}.
\end{equation}

We are interested in bounds on the number $N$ of points needed for the volume $|K_N|$ of $K_N$ to be asymptotic in expectation to the volume $|\conv K|$ of the convex hull of $K$ as $n \to \infty$. In the pioneering work \cite{Dy}, Dyer, F\"uredi and McDiarmid established sharp thresholds for the vertices of the cube $K = \{-1,1\}^n$, as well as for the solid cube $K = [-1,1]^n$. More precisely, they showed that for every $\varepsilon > 0$,
\begin{equation}\label{eq:cube-thre}
\frac{\E|K_N|}{2^n} \xrightarrow[n \to \infty]{} \begin{cases} 0, & \text{ if }N \leq (\nu - \varepsilon)^n \\ 1, & \text{ if }N \geq (\nu + \varepsilon)^n \end{cases}\,,
\end{equation}
where for $K = \{-1,1\}^n$, we have $\nu = 2/\sqrt{e} = 1.213...$ and for $K = [-1,1]^n$, we have $\nu = 2\pi e^{-\gamma-1/2} = 2.139...$ (see also \cite{FS}). For further generalisations establishing sharp exponential thresholds see \cite{Gian} (in a situation when the $X_i$ are not uniform on a set but have i.i.d. components compactly supported in an interval).

The case of a Euclidean ball is different. Pivovarov showed in \cite{Piv} (see also \cite{beta1}) that when 
\[
K = B_2^n= \{x \in \R^n, \ \sum x_i^2 \leq 1\},\]
 the threshold is superexponential, that is for every $\varepsilon > 0$,
\begin{equation}\label{eq:ball-thre}
\frac{\E|K_N|}{|K|} \xrightarrow[n \to \infty]{} \begin{cases} 0, & \text{ if }N \leq e^{(1-\varepsilon)\cdot\frac{1}{2}n\log n} \\ 1, & \text{ if }N \geq e^{(1+\varepsilon)\cdot\frac{1}{2}n\log n} \end{cases}\, .
\end{equation}
He additionally considered the situation when the $X_i$ are not uniform on a set but are Gaussian. 

In recent works \cite{beta1, beta2}, the authors study the case of the $X_i$ having rotationally invariant densities of the form ${\red \text{const} \cdot}(1-\sum x_i^2)^{\beta}\1_{B_2^n}$, $\beta > -1$. This is the so-called Beta model of random polytopes attracting considerable attention in stochastic geometry. In particular, $\beta = 0$ corresponds to the uniform distribution on the unit ball and the limiting case $\beta \to -1$ corresponds to the uniform distribution on the unit sphere. As established in \cite{beta1}, the threshold here is as follows: for every constant $\varepsilon \in (0,1)$ and sequences $N = N(n)$, $-1 < \beta = \beta(n)$, we have
\begin{equation}\label{eq:beta-thre}
\frac{\E|K_N|}{|B_2^n|} \xrightarrow[n \to \infty]{} \begin{cases} 0, & \text{ if }N \leq e^{(1-\varepsilon)(\frac{n}{2}+\beta)\log n} \\ 1, & \text{ if }N \geq e^{(1+\varepsilon)(\frac{n}{2}+\beta)\log n} \end{cases}\,,
\end{equation}
which was further refined in \cite{beta2}: for every positive constant $c$, the limit is $e^{-c}$ if $N$ grows like $e^{(\frac{n}{2}+\beta)\log\frac{n}{2c}}$ as $n\to\infty$.

We would like to focus on establishing general bounds for some large natural families of distributions. Specifically, suppose that for each dimension $n$, we are given a family $\{\mu_{n,i}\}_{i \in I_n}$ of probability measures such that each $\mu_{n,i}$ is supported on a compact set $V_{n,i}$ in $\R^n$. We would like to find the largest number $N_0$ and the smallest number $N_1$ (in terms of $n$ and some parameters of the family) such that for every $\mu_{n,i}$ from the family, $\frac{\E|K_N|}{|\mathrm{conv} V_{n,i}|} = o(1)$ for $N \leq N_0$ and $\frac{\E|K_N|}{|\mathrm{conv} V_{n,i}|} = 1-o(1)$ for $N \geq N_1$ as $n \to \infty$ ($K_N$ is a random polytope given by \eqref{eq:def-KN} with $X_1, X_2,\ldots$ being i.i.d. drawn from $\mu_{n,i}$).

For instance, the examples of the cube and the ball suggest that for the family of uniform measures on convex bodies, $N_0$ is exponential and $N_1$ is super-exponential in~$n$. 

In fact, the latter can be quickly deduced from a classical result by Groemer from \cite{Gr}, combined with the thresholds for Euclidean balls established by Pivovarov in \cite{Piv}. Groemer's theorem says that for every $N > n$, we have 
\[
\E|\conv\{X_1,\ldots,X_N\}| \geq \E|\conv\{Y_1,\ldots,Y_N\}|,
\]
where the $X_i$ are i.i.d. uniform on a convex set $K$ and the $Y_i$ are i.i.d. uniform on a Euclidean ball with the same volume as $K$. We thus get from \eqref{eq:ball-thre} that
\begin{equation}\label{eq:grom-thre}
{\red \frac{1}{|K|}}\E|\conv\{X_1,\ldots,X_N\}| = 1-o(1),
\end{equation}
as long as $N \geq e^{(1+\varepsilon)\frac{n}{2}\log n}$.

In this work, we shall establish an exponential bound on $N_0$ for the family of log-concave distributions on convex sets and extend \eqref{eq:grom-thre} to the family of the so-called $\kappa$\nobreakdash-concave distributions.

\paragraph{Acknowledgements.} We would like to thank Alan Frieze for many helpful discussions. {\red We are indebted to an anonymous referee for carefully reading our manuscript and for a very helpful report significantly improving the paper.}


\section{Results}

Recall that a Borel probability measure $\mu$ on $\R^n$ is $\kappa$-concave, $\kappa \in [-\infty,\frac{1}{n}]$, if for every $\lambda \in [0,1]$ and every Borel sets $A$, $B$ in $\R^n$, we have
\[
\mu(\lambda A + (1-\lambda) B) \geq \Big(\lambda \mu(A)^{\kappa} + (1-\lambda)\mu(B)^{\kappa}\Big)^{1/\kappa}
\]
{\red (for background on $\kappa$-concave measures see e.g. \cite{Bor1, Bor2} or Section 2.1.1 in \cite{Gian2}).}
We say that a random vector is $\kappa$-concave if its law is $\kappa$-concave. For example, vectors uniform on convex bodies in $\R^n$ are $1/n$-concave {\red by the Brunn-Minkowski inequality}. The right hand side increases with $\kappa$, so as $\kappa$ increases, the class of $\kappa$-concave measures becomes smaller. It is a natural extension of the class of log-concave random vectors, corresponding to $\kappa = 0$, with the right hand side in the defining inequality understood as the limit $\kappa\to 0+$. Many results for convex sets have analogues for concave measures (for instance, see \cite{B1, B2, BM0, FGP, Gue}). 

Consider $\kappa \in (0,1/n)$. {\red By Borell's theorem from \cite{Bor1}}, a $\kappa$-concave random vector is supported on a convex body, {\red has a density} and its density is a $1/\beta$-concave function, that is of the form $h^{\beta}$ for a concave function $h$ with $\beta = \kappa^{-1}-n$. The notion of $\kappa$-concavity was introduced and studied by Borell in \cite{Bor1, Bor2}, which are standard references on this topic. 

We also recall that a random vector $X$ in $\R^n$ is isotropic if it is centred, that is $\E X = 0$ and its covariance matrix $\text{Cov}(X) = [\E X_iX_j]_{i,j \leq n}$ is the identity matrix. The isotropic constant $L_X$ of a log-concave random vector $X$ {\red which is isotropic and has} density $f$ {\red on $\R^n$} is defined as $L_X = (\text{ess sup}_{\R^n} f)^{1/n}$ (see e.g. \cite{Gian2}). {\red By Borell's theorem, every log-concave random vector in $\R^n$ is supported on an affine subspace of $\R^n$ and has a density with respect to Lebesgue measure on that subspace.}

Our first main result {\red suggests a necessary condition on $N$ (in the form of a lower bound for $N$ exponential in the dimension $n$) so that $\E|K_N|$ will be significant in the case} of symmetric log-concave distributions supported in convex bodies. {\red We recall that a measure $\mu$ on $\R^n$ is symmetric (sometimes also called even) if $\mu(A) = \mu(-A)$ for every $\mu$-measurable set $A$ in $\R^n$.}

\begin{theorem}\label{thm:exp-notenough}
Let $\mu$ be a symmetric log-concave probability measure supported on a convex body $K$ in $\R^n$. Let $X_1,X_2,\ldots$ be i.i.d. random vectors distributed according to $\mu$. Let $K_N = \conv\{X_1,\ldots,X_N\}$. There are universal positive constants $c_1, c_2$ such that if $N \leq e^{c_1n/L_\mu^2}$, then 
\[
\frac{\E|K_N|}{|K|} \leq e^{-c_2n/L_\mu^2},
\]
where $L_\mu$ is the isotropic constant of $\mu$.
\end{theorem}

Our second main result provides a {\red sufficient condition on $N$ so that $\E|K_N|$ will be significant in the case} 
 of $\kappa$-concave distributions. 

\begin{theorem}\label{thm:superexp-enough}
Let $\mu$ be a symmetric $\kappa$-concave measure on $\R^n$ with $\kappa \in (0,\frac{1}{n})$, supported on a convex body $K$ in $\R^n$. Let $X_1,X_2,\ldots$ be i.i.d. random vectors uniformly distributed according to $\mu$. Let $K_N = \conv\{X_1,\ldots,X_N\}$. There is a universal constant $C$ such that for every $\omega > C$, if $N \geq e^{\frac{1}{\kappa}(\log n + 2\log\omega)}$, then 
\[
\frac{\E|K_N|}{|K|} \geq 1-\frac{1}{\omega}.
\]
\end{theorem}

\section{{\red Floating bodies}}

It turns out that the following quasi-concave function plays a crucial role in estimates for the expected volume of the convex hull of random points (see \cite{Bar-surv, Bar1, Dy}): for a random vector $X$ in $\R^n$ define 
\begin{equation}\label{eq:def-q}
q_X(x) = \inf\{\p{X \in H}, \ H \text{ half-space containing $x$}\}, \qquad x \in \R^n.
\end{equation}
It is clear that $q(\lambda x + (1-\lambda) y) \geq \min\{q(x), q(y)\}$, because if a half-space $H$ contains $\lambda x + (1-\lambda)y$, it also contains $x$ or $y$. Consequently, superlevel sets 
\begin{equation}\label{eq:L-def}
L_{q_X, \delta} = \{x \in \R^n, \ q_X(x) \geq \delta\}
\end{equation}
of this function are convex. Another way of looking at these sets is by noting that they are intersections of half-spaces: $L_{q_X, \delta} = \bigcap\{H: \,\text{$H$ is a half-space},\, \p{X \in H} > 1- \delta\}$. When $X$ is uniform on a convex set $K$, they are called convex floating bodies ($K \setminus L_{q_X, \delta}$ is called a wet part). The function $q_X$ in statistics is called the Tukey or half-space depth of $X$. The two notions have been recently surveyed in \cite{Schu}. 

A key lemma from \cite{Dy} relates the volume of random convex hulls of i.i.d. samples of $X$ to the volume of the level sets $L_{q_X, \delta}$.  Bounds on the latter are obtained by a combination of elementary convexity arguments and deep results from asymptotic convex geometry (notably, Paouris' reversal of the $L_p$-affine isoperimetric inequality due to Lutwak, Yang and Zhang). We shall present these and all the necessary background material in Section \ref{sec:aux}. Section \ref{sec:proof} is devoted to our proofs.

\section{Auxiliary results}\label{sec:aux}

\subsection{Log-concave and $\kappa$-concave measures}\label{sec:kappa-conc}

Theorem 4.3 from \cite{Bor2} provides in particular the following stability of $\kappa$-concavity with respect to taking marginals: if $\kappa \in (0,\frac1n)$ and $f$ is the density of a $\kappa$-concave random vector in $\R^n$, then
\begin{equation}\label{eq:marg}
\text{the marginal } \ x \mapsto \int_{\R^{n-1}} f(x,y) \dd y \ \text{ is a } \frac{\kappa}{1-\kappa}\text{-concave function}.
\end{equation}
We will also need the following basic estimate: if $g\colon \R\to[0,+\infty)$ is the density of a log-concave random variable $X$ with $\E X = 0$ and $\E X^2 = 1$, then
\begin{equation}\label{eq:log-c}
\frac{1}{2\sqrt{3}e} \leq g(0) \leq \sqrt{2}
\end{equation}
(see e.g. Chapter 10.6 in \cite{Gian1}).

\subsection{Central lemma}\label{sec:central-lemma}

{\red The idea of using floating bodies to estimate volume of random polytopes goes back to \cite{Bar1}}. The following is a key lemma from \cite{Dy} (called by the authors ``central'') about asymptotically matching upper and lower bounds for the volume of the random convex hull.

\begin{lemma}[\cite{Dy}]\label{lm:central}
Suppose $X_1, X_2,\ldots$ are i.i.d. random vectors in $\R^n$. Let $K_N = \conv\{X_1,\ldots,X_N\}$ and define $q = q_{X_1}$ by \eqref{eq:def-q}.
Then for every {\red Borel} subset $A$ of $\R^n$, we have
\begin{equation}\label{eq:upp-bd}
\E|K_N| \leq |A| + N\cdot \left(\sup_{A^c} q\right) \cdot |A^c \cap \{x \in \R^n, \ q(x) > 0\}|
\end{equation}
and, {\red if additionally $\mu$ assigns zero mass to every hyperplane in $\R^n$}, then
\begin{equation}\label{eq:low-bd}
\E|K_N| \geq |A|\left(1 - 2\binom{N}{n}\left(1 - \inf_A q\right)^{N-n}\right).
\end{equation}
\end{lemma}
(The proof therein concerns only the cube, but their argument repeated verbatim justifies our general situation as well -- see also \cite{Gian}.)

\subsection{Bounds related to the function $q$}

Lemma \ref{lm:central} is applied to level sets $L_{q,\delta}$ of the function $q$ (see \eqref{eq:L-def}). We gather here several remarks concerning bounds for the volume of such sets. For the upper bound, we will need the containment $L_{q,\delta} \subset cZ_{\alpha}(X)$, where $c$ is a universal constant and $Z_{\alpha}$ is the centroid body (defined below). This was perhaps first observed in Theorem 2.2 in \cite{PW} (with a reverse inclusion as well). We recall an argument below. 

\begin{remark}\label{rem:formula-q}
Plainly, for the infimum in the definition \eqref{eq:def-q} of $q_X(x)$, it is enough to take half-spaces for which $x$ is on the boundary, that is
\begin{equation}\label{eq:fromula-q}
q_X(x) = \inf_{\theta \in \R^n} \p{ \scal{X-x}{\theta} \geq 0},
\end{equation}
where $\scal{u}{v} = \sum_i u_iv_i$ is the standard scalar product in $\R^n$. Of course, by homogeneity, this infimum can be taken only over unit vectors. We also remark that by Chebyshev's inequality,
\[
\p{\scal{X-x}{\theta}\geq 0} \leq e^{-\scal{\theta}{x}}\E e^{\scal{\theta}{X}}.
\]
Consequently,
\[
q_X(x) \leq \exp\left(-\sup_{\theta \in \R^n}\left\{\scal{\theta}{x} - \log\E e^{\scal{\theta}{X}} \right\}  \right)
\]
and we have arrived at the Legendre transform $\Lambda_X^\star$ of the log-moment generating function $\Lambda_X$ of $X$,
\[
\Lambda_X(x) = \log \E e^{\scal{x}{X}} \qquad \text{and} \qquad \Lambda_X^\star(x) = \sup_{\theta \in \R^n} \left\{\scal{\theta}{x} - \Lambda_X(\theta)\right\}.
\]
Thus, for every $\alpha > 0$, we have
\begin{equation}\label{eq:q-subset-Bt}
\{x \in \R^n, \ q_X(x) > e^{-\alpha}\} \subset \{x \in \R^n, \ \Lambda_X^\star(x) < \alpha\}.
\end{equation}
\end{remark}

\begin{remark}\label{rem:Bt-to-Zp}
The level sets $\{\Lambda_X^\star < \alpha\}$ have appeared in a different context of the so-called optimal concentration inequalities introduced by Lata\l a and Wojtaszczyk in \cite{LW}. Modulo universal constants, they turn out to be equivalent to centroid bodies playing a major role in asymptotic convex geometry (see \cite{LZ, P1, P2, P4, P5}). Specifically, for a random vector $X$ in $\R^n$ and $\alpha \geq 1$, we define its $L_\alpha$-centroid body $Z_\alpha(X)$ by 
\[
Z_\alpha(X) = \{x \in \R^n, \ \sup\{\scal{x}{\theta}, \ \E|\scal{X}{\theta}|^\alpha \leq 1\} \leq 1\}
\]
(equivalently, the support function of $Z_\alpha(X)$ is $\theta \mapsto (\E|\scal{X}{\theta}|^\alpha)^{1/\alpha}$).
By Propositions 3.5 and 3.8 from \cite{LW}, if $X$ is a symmetric log-concave random vector $X$ (in particular, uniform on a symmetric convex body),
\begin{equation}\label{eq:q-in-Zp}
\{\Lambda_X^* < \alpha\} \subset 4eZ_\alpha(X), \qquad \alpha \geq 2.
\end{equation}
(A reverse inclusion $Z_\alpha(X) \subset 2^{1/\alpha}e\{\Lambda_X^* < \alpha\}$ holds for any symmetric random vector, see Proposition 3.2 therein.) 

We shall need an upper bound for the volume of centroid bodies. This was done by Paouris (see \cite{P4}). Specifically, Theorem 5.1.17 from \cite{Gian2} says that if $X$ is an isotropic log-concave random vector in $\R^n$, then
\begin{equation}\label{eq:Paouris}
|Z_\alpha(X)|^{1/n} \leq C\sqrt{\frac{\alpha}{n}}, \qquad 2 \leq \alpha \leq n,
\end{equation} 
where $C$ is a universal constant.
\end{remark}

\begin{remark}
Significant amount of work in \cite{Dy} was devoted to showing that, for the cube, inclusion \eqref{eq:q-subset-Bt} is nearly tight (for \emph{correct} values of $\alpha$, using exponential tilting of measures typically involved in establishing large deviation principles). We shall take a different route and put a direct lower bound on $q_X$ described in the following lemma. Our argument is based on property \eqref{eq:marg}. 
\end{remark}

\begin{lemma}\label{lm:q-lowbd}
Let $\kappa \in (0,\frac{1}{n})$. Let $X$ be a symmetric isotropic $\kappa$-concave random vector supported on a convex body $K$ in $\R^n$. Then for every unit vector $\theta$ in $\R^n$ and $a > 0$, we have 
\begin{equation}\label{eq:q-lowbd-gen}
\p{\scal{X}{\theta} > a} \geq \frac{1}{16}\kappa\left(1 - \frac{a}{h_K(\theta)}\right)^{1/\kappa},
\end{equation}
where $h_K(\theta) = \sup_{y \in K} \scal{y}{\theta}$ is the support function of $K$.
In particular, denoting the norm given by $K$ as $\|\cdot\|_K$, we have
\begin{equation}\label{eq:q-lowbd}
q_X(x) \geq \frac{1}{16}\kappa\left(1 - \|x\|_K\right)^{1/\kappa}, \qquad x \in K.
\end{equation}
\end{lemma}
\begin{proof}
Consider the density $g$ of $\scal{X}{\theta}$. Let $b = h_K(\theta)$. Note that $g$ is supported in $[-b,b]$.
By \eqref{eq:marg}, $g^{\frac{\kappa}{1-\kappa}}$ is concave, thus on $[0,b]$ we can lower-bound it by a linear function whose values agree at the end points,
\[
g(t)^{\frac{\kappa}{1-\kappa}} \geq g(0)^{\frac{\kappa}{1-\kappa}}\left(1 - \frac{t}{b}\right), \qquad t \in [0,b].
\]
This gives
\[
\p{\scal{X}{\theta} > a} = \int_a^{b} g(t) \dd t \geq g(0)\int_a^{b} \left(1 - \frac{t}{b}\right)^{\frac{1-\kappa}{\kappa}} \dd t = \kappa g(0)b\left(1 - \frac{a}{b}\right)^{1/\kappa}.
\]
Since $\scal{X}{\theta}$ is in particular log-concave, by \eqref{eq:log-c}, we have $\frac{1}{2\sqrt{3}e} \leq g(0) \leq \sqrt{2}$. Moreover, by isotropicity,
\[
1 = \E\scal{X}{\theta}^2 = \int_{-b}^b t^2g(t) \dd t \leq \,b^2\!\int_{-b}^b g(t)\dd t = b^2.
\]
Thus, say $g(0)b > \frac{1}{16}$ and we get \eqref{eq:q-lowbd-gen}.
To see \eqref{eq:q-lowbd}, first recall \eqref{eq:fromula-q}. By symmetry, $\p{ \scal{X-x}{\theta} \geq 0} = \p{ \scal{X}{\theta} \geq |\scal{x}{\theta}|}$, so we use \eqref{eq:q-lowbd-gen} with $a = |\scal{\theta}{x}|$ and note that by the definition of $h_K$, $\left|\scal{\frac{x}{\|x\|_K}}{\theta}\right| \leq h_K(\theta)$, so $\frac{|\scal{x}{\theta}|}{h_K(\theta)} \leq \|x\|_K$.
\end{proof}

\section{Proofs}\label{sec:proof}

\subsection{Proof of Theorem \ref{thm:exp-notenough}}

Since the quantity $\frac{\E|K_N|}{|K|}$ does not change under invertible linear transformations applied to $\mu$, without loss of generality we can assume that $\mu$ is isotropic.  Let $q = q_{X_1}$ be defined by \eqref{eq:def-q}. Fix $\alpha > 0$ and apply \eqref{eq:upp-bd} to the set $A = \{x, \ q(x) > e^{-\alpha}\}$. We get
\[
\frac{\E|K_N|}{|K|} \leq \frac{|A|}{|K|} + Ne^{-\alpha}
\]
(we have used $\{x, \ q(x) > 0\} \subset K$ to estimate the last factor in \eqref{eq:upp-bd} by $1$). Combining \eqref{eq:q-subset-Bt}, \eqref{eq:q-in-Zp} and \eqref{eq:Paouris},
\[
|A| \leq |4eZ_{\alpha}(X)| \leq \left(4eC\sqrt{\frac{\alpha}{n}}\right)^n.
\]
Moreover, by the definition of the isotropic constant of $\mu$,
\[
1 = \int_K \dd \mu \leq L_\mu^n|K|.
\]
Thus,
\[
\frac{|A|}{|K|} \leq \left(4eCL_\mu\sqrt{\frac{\alpha}{n}}\right)^n.
\]
We set $\alpha$ such that $4eCL_\mu\sqrt{\frac{\alpha}{n}} = e^{-1}$ and adjust the constants to finish the proof. $\square$

\subsection{Proof of Theorem \ref{thm:superexp-enough}}

As in the proof of Theorem \ref{thm:exp-notenough}, we can assume that $\mu$ is isotropic. Let $q = q_{X_1}$ be defined by \eqref{eq:def-q}.  
Consider $0< \beta < 1$ (to be fixed shortly). By \eqref{eq:low-bd} which we apply to the set $A = \{x \in K, \ q(x) > \beta^{1/\kappa}\}$, we have 
\[
\frac{\E|K_N|}{|K|} \geq \frac{|A|}{|K|}\left(1 - 2\binom{N}{n}\left(1 - \beta^{1/\kappa}\right)^{N-n}\right).
\]
{\red (The extra assumption needed in \eqref{eq:low-bd} is satisfied: by Borell's theorem from \cite{Bor1}, $\mu$ has a density on its support which by our assumption is $n$-dimensional, hence $\mu(H) = 0$ for every hyperplane $H$ in $\R^n$.)} 
By the lower bound on $q$ from \eqref{eq:q-lowbd},
\[
A \supset \{x \in \R^n, \ \|x\|_K \leq 1 - (16\kappa^{-1})^\kappa\beta\},
\]
hence, as long as $(16\kappa^{-1})^\kappa\beta < 1$,
\[
\frac{|A|}{|K|} \geq \left(1 - (16\kappa^{-1})^\kappa\beta\right)^n \geq 1 - n(16\kappa^{-1})^\kappa\beta \geq 1 - 32n\beta.
\]
We choose $\beta$ such that $32n\beta = \frac{1}{2\omega}$ and crudely deal with the second term,
\[
\binom{N}{n}\left(1 - \beta^{1/\kappa}\right)^{N-n} \leq N^ne^{-\beta^{1/\kappa}(N-n)},
\]
which is nonincreasing in $N$ as long as $N \geq n\beta^{-1/\kappa}$. This holds for $\omega$ large enough if, say $N \geq n^{1/\kappa}\omega^{2/\kappa}$. Then we easily conclude that the dominant term above is $e^{-\beta^{1/\kappa}N}$ which yields, say
\[
\frac{\E|K_N|}{|K|} \geq \left(1 - \frac{1}{2\omega}\right)(1-2e^{-\omega^{n/2}}) \geq 1 - \frac{1}{\omega},
\]
provided that $n$ and $\omega$ are large enough. $\square$

\section{Final remarks}

\begin{remark}
Groemer's result used in \eqref{eq:grom-thre} for uniform distributions has been substantially generalised by Paouris and Pivovarov in \cite{PP} to arbitrary distributions with bounded densities. We remark that in contrast to \eqref{eq:grom-thre}, using the extremality result of the ball from \cite{PP} does not seem to help obtain bounds from Theorem \ref{thm:superexp-enough} for two reasons. For one, it concerns bounded densities and rescaling will cost an exponential factor. Moreover, for the example of $\beta$-polytopes from \cite{beta1}, we have that they are generated by $\kappa$-concave measures with $\kappa = \frac{1}{\beta+n}$ and the sharp threshold for the volume is of the order $n^{(\beta+n/2)}$ (see \eqref{eq:ball-thre}). The ball would give that $N_1 = n^{(1+\varepsilon)n/2}$ points is enough. 
\end{remark}

\begin{remark}
The example of beta polytopes from \eqref{eq:ball-thre} shows that the bound on $N$ in Theorem \ref{thm:superexp-enough} has to be at least of the order $n^{\beta + n/2} = n^{\frac{1}{\kappa}-n/2} \geq n^{\frac{1}{2\kappa}}$. Our bound $n^{\frac{1}{\kappa}}$ is perhaps suboptimal. It is not inconceivable that as in the uniform case, the extremal example is supported on a Euclidean ball. 
\end{remark}

\begin{remark}
It is reasonable to ask about sharp thresholds like the ones in \eqref{eq:cube-thre}, \eqref{eq:ball-thre} and \eqref{eq:beta-thre} for other sequences of convex bodies, say simplices, cross-polytopes, or in general $\ell_p$-balls. This is a subject of ongoing work. We refer to \cite{FPT} for recent results establishing exponential nonsharp thresholds for a simplex (i.e. with a gap between the constants for lower and upper bounds).
\end{remark}

\end{document}